		\documentclass[10pt]{elsarticle}
		\usepackage{etoolbox}
		\makeatletter
		\patchcmd{\ps@pprintTitle}{\footnotesize\itshape
		Preprint submitted to \ifx\@journal\@empty Elsevier
		\else\@journal\fi\hfill\today}{\relax}{}{}
		\makeatother
\usepackage{lineno,hyperref}
\modulolinenumbers[5]
\usepackage{amsfonts}
\usepackage{amssymb}
\usepackage{amsmath}
\usepackage{amsthm}
\usepackage{graphicx}
\usepackage{stackrel}
\usepackage{filecontents}
\usepackage{enumerate}
\newtheorem{theorem}{Theorem}[section]
\newtheorem{lemma}{Lemma}[section]

\journal{arXiv}
\def\tr{{\rm{tr}}}

\begin{document}

\begin{frontmatter}
\title{A May--Holling--Tanner predator-prey model with multiple Allee effects on the prey and an alternative food source for the predator}
\author[QUTaddress,UDLAaddress]{Claudio Arancibia--Ibarra}
\author[USDaddress]{Jos\'e Flores}
\author[QUTaddress]{Michael Bode}
\author[QUTaddress]{Graeme Pettet}
\author[QUTaddress]{Peter van Heijster}
\address[QUTaddress]{School of Mathematical Sciences, Queensland University of Technology, \\ 
GPO Box 2434, GP Campus, Brisbane, Queensland 4001 Australia\\
claudio.arancibia@hdr.qut.edu.au}
\address[UDLAaddress]{Facultad de Educaci\'on, Universidad de Las Am\'ericas,\\
Av. Manuel Montt 948, Santiago, Chile}
\address[USDaddress]{Department of Computer Science, The University of South Dakota,\\
Vermillion, SD 57069, South Dakota, USA}
\begin{abstract}
We study a predator-prey model with Holling type I functional response, an alternative food source for the predator, and multiple Allee effects on the prey. We show that the model has at most two equilibrium points in the first quadrant, one is always a saddle point while the other can be a repeller or an attractor. Moreover, there is always a stable equilibrium point that corresponds to the persistence of the predator population and the extinction of the prey population. Additionally, we show that when the parameters are varied the model displays a wide range of different bifurcations, such as saddle-node bifurcations, Hopf bifurcations, Bogadonov-Takens bifurcations and homoclinic bifurcations. We use numerical simulations to illustrate the impact changing the predation rate, or the non-fertile prey population, and the proportion of alternative food source have on the basins of attraction of the stable equilibrium point in the first quadrant (when it exists). In particular, we also show that the basin of attraction of the stable positive equilibrium point in the first quadrant is bigger when we reduce the depensation in the model.
\end{abstract}
\end{frontmatter}

\newpage
\section{Introduction}
The goal of analysing the dynamics of complex ecological systems is to better describe the different interactions between species, to understand their longterm behaviour, and to predict how they will respond to management interventions \cite{hooper,may2}. Current predator-prey dynamics studies often use nonlinear mathematical models to describe the species’ interactions and answer these questions. These models aim to be representative of real natural phenomena and they should capture the essentials of the dynamics. However, new theoretical, empirical, and observational research in ecology is revealing species’ interactions to be much more complicated than previous models admit \cite{hanski2, hanski,roux,turchin}. Moreover, it is becoming increasingly apparent that our understanding of ecosystem dynamics will depend \cite{wood}, to some extent, on the particular nature of these interaction processes, such as the functional response or predation rate \cite{bimler,santos,turchin}. 

The standard approach for using models to understand ecological systems is to construct a model from first principles, and then compare species’ abundance timeseries to the predictions from those models. However, this approach becomes more difficult when we add additional nuance to standard models, making them more complex, more nonlinear, and more difficult to parameterise. For instance, Graham and Lambin \cite{graham} showed that field-vole (\textit{Microtus agrestis}) survival can be affected by reducing the weasel predation. They also demonstrated that weasel proportion was suppressed in summer and autumn, while the voles (\textit{Microtus agrestis}) population always declined to low density. However, they argued that the underlying model was too hard to study due to the large number of parameters. Some ecologists have attempted to resolve this issue by applying qualitative approaches, which make few assumptions about the models’ functional forms or parameters \cite{baker,levins,Raymond}. However, we can also approach the problem by trying to understand the topology of the associated dynamical system, rather than specific trajectories \cite{prager}. Such a topological approach may offer general and global insights into the behaviour of the system without requiring accurate parameter estimates. 

The phenomena described above can be observed in predator-prey theory. The original Lotka-Volterra predator prey models \cite{lotka} were straightforward, with simple functional forms for species’ growth and interactions. Empirical observations required successive changes to these assumptions, leading inter alia to the May--Holling--Tanner model, which is itself a special case of the Leslie--Gower predator-prey model \cite{hsu,saez}. The May--Holling--Tanner model is described by an autonomous two-dimensional system of ordinary differential equations, where the equations for the growth of the predator and prey are logistic-type functions, where the predator carrying capacity is a prey dependent \cite{aguirre1,flores2,turchin}. The functional response describing the predation is Holling Type I, which, for instance, models filter feeders where searching for food can occur at the same time that the species processes the food \cite{liermann}. A Holling Type I response function corresponds to a linear increasing function in the prey $H(x)=qx$ \cite{holling}. This type of functional response is also called Lotka-Volterra type. In particular, the model is given by
\begin{equation}\label{ht1}
\begin{aligned}
\dfrac{dx}{dt} &=	 rx\left(1-\dfrac{x}{K} \right)  - qxy\,,  \\
\dfrac{dy}{dt} &=	 sy\left(1-\dfrac{y}{nx} \right) \,. 
\end{aligned}
\end{equation}
Here, $x(t)$ and $y(t)$ represent the proportion of the prey respectively predator population at time $t$; $r$ is the intrinsic growth rate for the prey; $s$ is the intrinsic growth rate for the predator; $q$ is the per capita predation rate; $K$ is the prey carrying capacity; $n$ is a measure of the quality of the prey as food for the predator; and $\widetilde{K}(x)=nx$ is the prey dependent carrying capacity of the predator.

However, even model \eqref{ht1} does not take into account that some predators act as generalists \cite{andersson,erlinge,hansson}. For instance, weasels (\textit{Mustela nivalis}) in the boreal forest region in Fennoscandia can switch to an alternative food source, although its population growth may still be limited by the fact that its preferred food, voles (\textit{Microtus agrestis}), are not available abundantly \cite{hanski,korobei,turchin}. This characteristic can be modelled by modifying the prey dependent carrying capacity of the predator \cite{aziz}. That is, in \eqref{ht1} 
\begin{equation}\label{food}
\widetilde{K}(x)=nx\quad\text{is replaced by}\quad \overline{K}(x)=nx+c,
\end{equation}
where we assumed that the alternative food source is constant, which in turns means that the predator proportion is small in compared to the alternative food source. Model \eqref{ht1} with \eqref{food} was studied in \cite{arancibia7,aziz}. It was shown that, in comparison to the original model \eqref{ht1}, there is an extra equilibrium point on the $y$-axis corresponding to the extinction of the prey but not the predator. Moreover, the nodependentn-negative parameter $c$ desingularises the origin of system \eqref{ht1}. 

Another effect that is not incorporated in \eqref{ht1} is the Allee effect \cite{allee2}. The Allee effect corresponds to a density-dependent phenomenon in which fitness growth initially increases as population density increases \cite{berec, courchamp, kramer, stephens}. This effect is usually modelled by adding a factor $(x-m)$ to the logistic function where $m$ is the minimum viable population \cite{allee2, verdy, zhao, zu} and $0<m<K$. With the Allee effect included in \eqref{ht1}
\begin{equation}\label{allee1}
L_o(x)=rx\left(1-\dfrac{x}{K}\right)\quad\text{is replaced by}\quad L_m(x)=rx\left( 1-\dfrac{x}{K}\right)\left(x-m\right).
\end{equation}
For $0<m<K$, the per-capita grow rate of the the prey population with the Allee effect included is negative, but increasing, for $x \in [0,m)$, and this is referred to as the strong Allee effect.  When $m \leq 0$, the per-capita growth rate is positive but increases at low prey population densities and this is referred to as the weak Allee effect \cite{berec,courchamp2}. Additionally, the Allee effect can also refer to a decrease in per capita fertility rate at low population densities or a phenomenon in which fitness, or population growth, increases as population density increases \cite{allee2, courchamp, kramer, stephens}. For instance, Ostfeld and Canhan \cite{ostfeld} found that the stabilisation of vole (\textit{Microtus agrestis}) populations in southeastern New York depends on the variation in reproductive rate and recruitment of the population. This effect is referred to as the multiple Allee effect \cite{berec}, sometimes also called the double Allee effect \cite{angulo,gonzalez5}. To incorporate this multiple Allee effect in \eqref{allee1} $L_m(x)$ is replaced by
\begin{equation}\label{allee2} 
L_{b}(x)=rx\left( 1-\dfrac{x}{K}\right)\left( \dfrac{1}{x+b}\right)\left(x-m\right).
\end{equation}
Here, $b$ is the non-fertile prey population and $0<m<K$ \cite{allee2, verdy, zhao, zu}. The per-capita growth rate for the logistic growth function, strong and weak Allee effect; and the multiple Allee effect are shown in Figure \ref{Fig1}. We observe that the multiple Allee effect reduces the region of depensation, that is,  the region where the per-capita growth rate is positive and growing, when compared to the strong Allee effect. This effect can be generated by the reduction of the probability of fertilisation at lower population density \cite{liermann}. This reduction commonly occurs in plants such as \textit{Diplotaxis erucoides}, \textit{Banksia goodii} and \textit{Clarkia concinna} \cite{liermann}.  IIn particular, the depensation region for the multiple Allee effect is given by $(m,x_1)$ with 
\begin{align}
&x_1=-b+\sqrt{(b+K)(b+m)} \,, \label{DA_D}
\end{align} 
and for the strong Allee effect by  $(m,x_2)$ with
\begin{align}
&x_2=\dfrac{1}{2}(K+m)\,,  \label{SA_D}
\end{align} 
and $x_1\leq x_2$ for all values of $b$, see Figure \ref{Fig1}.
\begin{figure}
\begin{center}
\includegraphics[width=12cm]{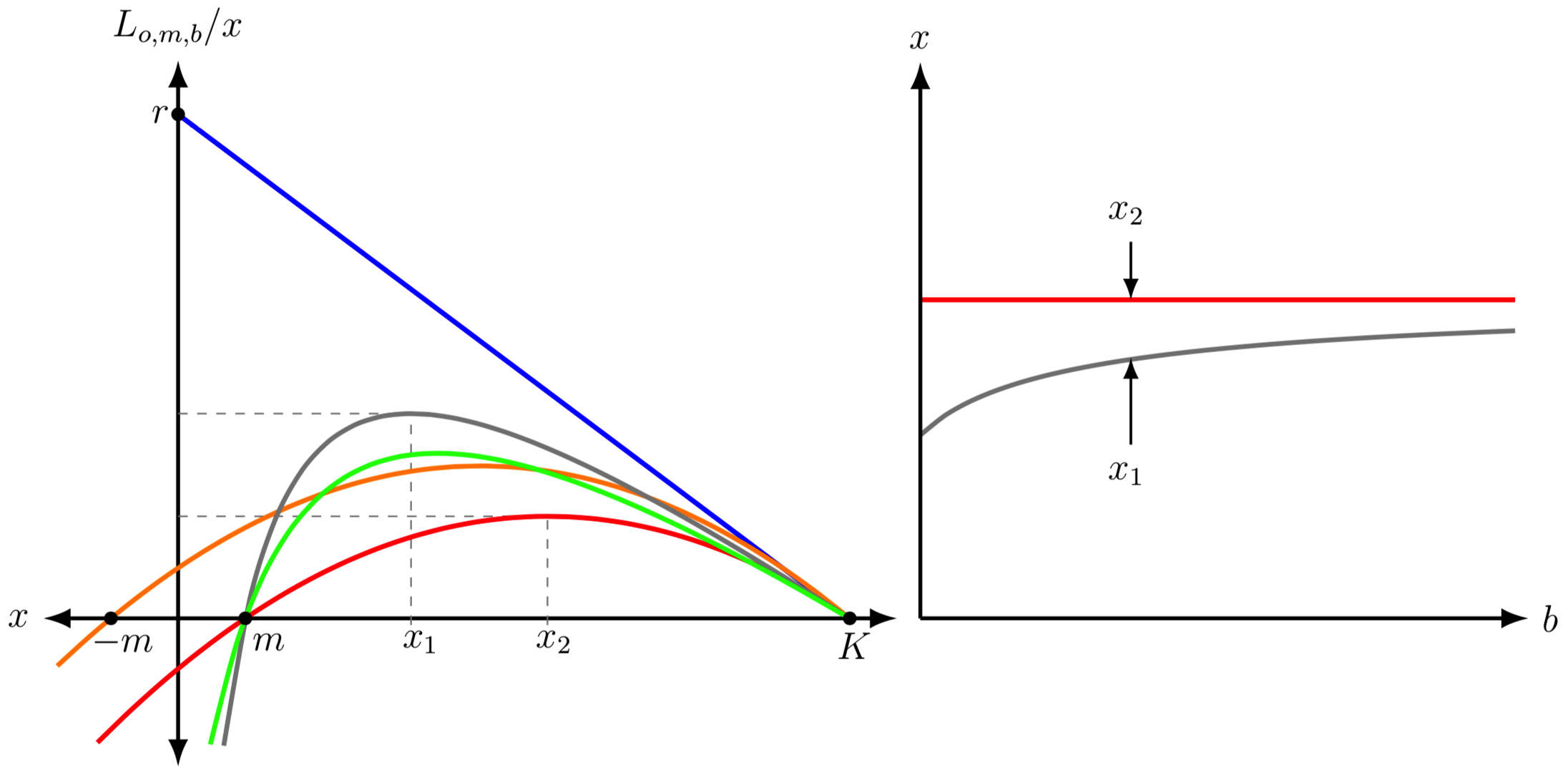}
\end{center}
\caption{In the left panel, we show the per capita growth rate of the logistic function (blue line), the strong Allee effect with $m = 0.1$ (red curve), the weak Allee effect with $m = -0.1$ (orange curve), multiple Allee effects with $m = 0.1$ and $b = 0.15$ (grey curve) and multiple Allee effects with $m = 0.1$ and $b = 0.05$ (green curve).
In the right panel, we show the size of the depensation region for the strong Allee effect \eqref{SA_D}  (red curve) and for the multiple Allee effects \eqref{DA_D} (grey curve) as function of the non-fertile prey population $b$. We observe that the depensation region for the multiple Allee effects is always smaller than the depensation region for the strong Allee effect.}
\label{Fig1}
\end{figure} 

When the alternative food \eqref{food} and the multiple Allee effect \eqref{allee2} are included in the modified May--Holling--Tanner model \eqref{ht1} it becomes
\begin{equation}\label{eq:04}
\begin{aligned}
\dfrac{dx}{dt} & = x\left( \dfrac{r}{x+b}\right)\left( 1-\dfrac{x}{K}\right)\left(x-m\right) -qxy,\\ 
\dfrac{dy}{dt} & = sy\left( 1\ -\dfrac{y}{nx+c}\right).
\end{aligned} 
\end{equation}

The aim of this manuscript is to study the dynamics of \eqref{eq:04} and, in particular, understanding the change in dynamics the multiple Allee effect and the alternative food source causes. Additionally, models \eqref{ht1} and \eqref{eq:04} without alternative food sources revealed that there exists a subset of the system parameters where the predator and prey population goes extinct \cite{martinez2}. However, these models assumed different dynamics at low abundance, and the absence of an alternative prey. We find that the alternative food source desingularises the origin and it prevents the extinction of the predator populations. Moreover, we study the basins of attraction of the stable positive equilibrium point(s) by modifying the predation rate $q$ and/or the alternative food source $c$. Moreover, we will show that the addition of the alternative food source and the multiple Allee effects will lead to complex dynamics, and different types of bifurcations such as Hopf bifurcations, homoclinic bifurcations, saddle-node bifurcations and Bogadonov-Takens bifurcations. This manuscript also extends the properties of the May--Holling--Tanner model with multiple Allee effects studied in \cite{martinez2} that is \eqref{eq:04} with $c=0$ by showing the impact of the inclusion of alternative food sources for predators. In addition, it complements the results of the May--Holling--Tanner model considering only alternative food for the predator studied in \cite{arancibia, flores} and the model considering only a single Allee effect on the prey and no alternative food for the predator studied in \cite{vanvoorn, yue2}. Model \eqref{ht1} with functional response Holling type II, i.e. $H(x) = qx/(x+a)$, was studied in \cite{hsu} where the authors showed that there is a region a parameter space where the unique positive equilibrium point is globally asymptotically stable. This model was also studied in \cite{saez} where the authors proved the existence of two limit cycles and the species can thus coexist and oscillate.
 
The basic properties of the model are briefly described in Section \ref{model}. In Section \ref{result} we prove the stability of the equilibrium points and give the conditions for the different types of bifurcations. In addition, we discuss the impact changing the predation rate or the alternative food source has on the basins of attraction of the positive equilibrium point in system \eqref{eq:04}. We further discuss the results and give the ecological implications in Section \ref{con}.

\section{The Model}\label{model}
Following \cite{arancibia7,arancibia, arancibia2, blows}, we introduce dimensionless variables $(u,v,\tau)$ by the function $\varphi :\breve{\Omega}\times\mathbb{R}\rightarrow \Omega\times\mathbb{R}$, where $\varphi(u,v,\tau)=(x,y,t)=(Ku,nKv,\tau(u+c/(nK))(u+b/K)/r)$, $\Omega=\{(x,y)\in \mathbb{R}^2, x\geq0, y\geq0\}$ and $\breve{\Omega}=\{(u,v)\in \mathbb{R}^2, u\geq0, v\geq0\}$. Additionally, we set $B:=b/K$, $C:=c/(nK)$, $M:=m/K\in(0,1)$, $S:=s/r$ and $Q:=qnK/r$, such that $(M,B,C,S,Q)\in \Pi= (0,1)\times\mathbb{R}^4_+$. This way, we convert \eqref{eq:04} to a  topologically equivalent nondimensionalised model given by
\begin{equation}\label{eq:05}
	\begin{aligned}
	\dfrac{du}{d\tau} & =  \ u(u+C)\left(\left(u-M\right)\left( 1-u\right) -Q(u+B)v\right), \\
	\dfrac{dv}{d\tau} & =  \ Sv\left(u+B\right)\left(u-v+C\right).
	\end{aligned} 
\end{equation}
The mapping $\varphi$ is a diffeomorphism which preserve the orientation of time since $\det\varphi(u,v,\tau)=nK^2u(u+b/K)/r>0$ \cite{chicone}. Therefore, system \eqref{eq:05} is topologically equivalent to system \eqref{eq:04} in $\Omega$. 
Furthermore, system \eqref{eq:05} is of Kolmogorov type since $du/d\tau=uR(u,v)$ and $dv/d\tau=vW(u,v)$, with $R(u,v)=(u+C)(u-M)(1-u)-Q(u+C)(u+B)v$ and $W(u,v)=S(u+B)(u-v+C)$. The $u$-nullcline of system \eqref{eq:05} in $\breve{\Omega}$ is $v=(u-M)(1-u)/Q(u+B)$, while the $v$-nullcline in $\breve{\Omega}$ is $v=u+C$. Hence, the equilibrium points in $\breve{\Omega}$ for the system \eqref{eq:05} are $(0,0)$, $(M,0)$, $(0,C)$ $(1,0)$ and $(u^*,v^*)$, where $u^*$ is determined by the roots of the following equation
\begin{align} \label{eq:06}
\begin{aligned}
& p(u):=(u-M)(1-u)=Q(u+C)(u+B)=:Qd(u),\quad\text{and}\quad v^*=u^*+C \,.
\end{aligned}
\end{align}
We observe that $\lim\limits_{u \rightarrow \pm \infty} p(u)=- \infty$ and $\lim\limits_{u \rightarrow \pm \infty} d(u)=\infty$. Hence, $p(u)$ can intersect $d(u)$ in the first quadrant  in two points; one point or not at all, see Figure \ref{Fig2}. 
\begin{figure}
\centering
\includegraphics[width=9cm]{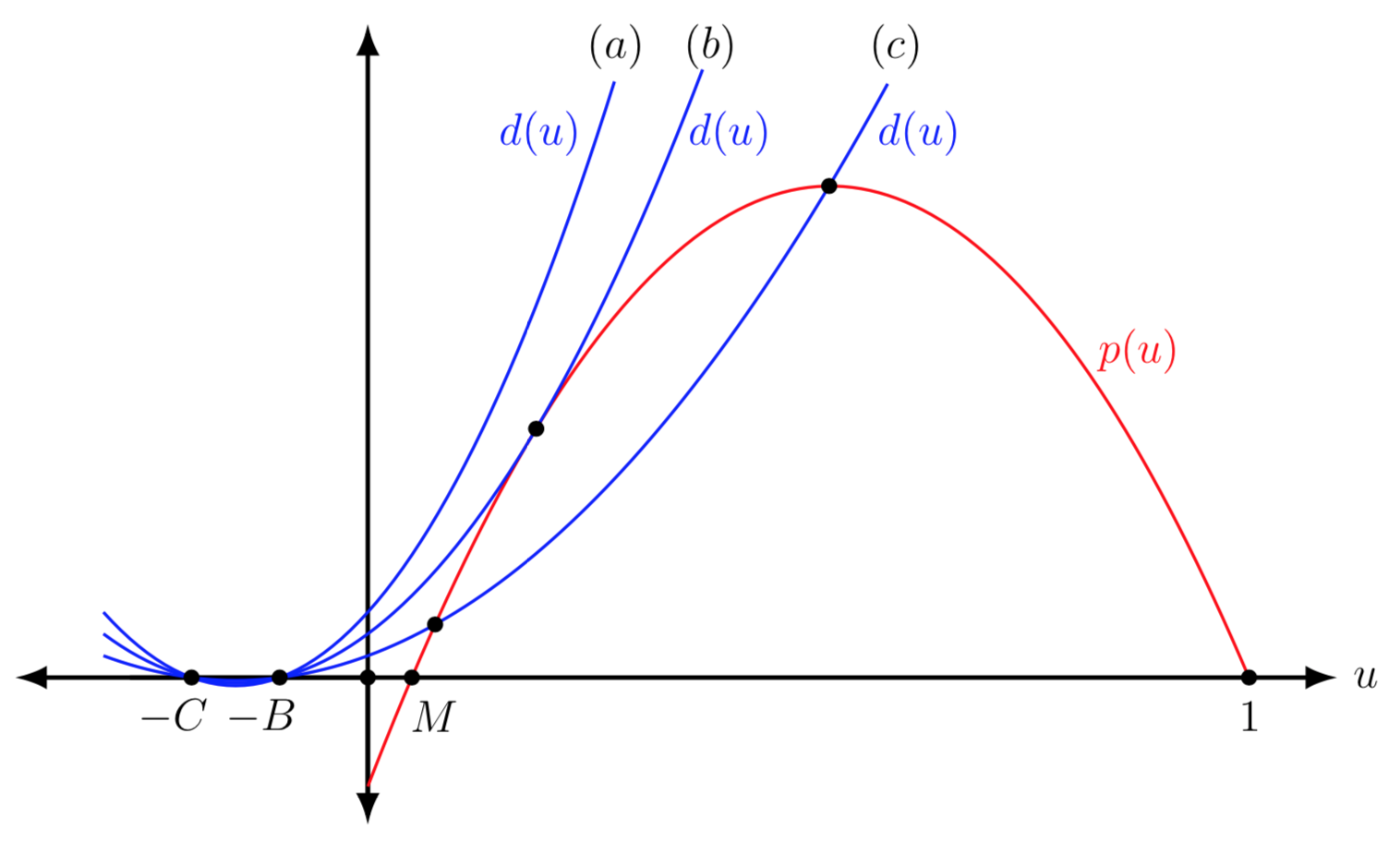}
\caption{The intersections of the functions $p(u)$ (red line) and $d(u)$ (blue lines) for three different possible cases: (a) If $\Delta<0$ \eqref{delta} then $p(u)$ and $d(u)$ do not intersect, and \eqref{eq:05} does not have positive equilibrium points; (b) If $\Delta=0$ then $p(u)$ and $d(u)$ intersect in one point, and \eqref{eq:05} has a unique positive equilibrium point; (c) If $\Delta>0$ then $p(u)$ and $d(u)$ intersect in two points, and \eqref{eq:05} has two distinct positive equilibrium points.}
\label{Fig2}
\end{figure}
The solutions of the equation \eqref{eq:06} are given by
\begin{align}\label{delta}
\begin{aligned}
u_{1,2}=\dfrac{1}{2(1+Q)}\left(1+M-Q(B+C)\pm\sqrt{\Delta}\right)\, \quad\text{with}\,\\
\Delta=(1+M-Q(B+C))^2-4(M+BCQ)(1+Q),
\end{aligned}
\end{align}
such that $M<u_1\leq u_3\leq u_2<1$, where $u_3=(1+M-Q(B+C))/(2(1+Q))$. That is, if \eqref{eq:06} has two real-valued solutions then these solutions are in the interval $(M,1)$.

Varying the parameters $Q$ and $C$ modifies the value of $\Delta$ and hence the number of equilibrium points in the first quadrant. Specifically:
\begin{enumerate}[(a)]
\item System \eqref{eq:05} has no positive equilibrium points if $\Delta<0$;
\item System \eqref{eq:05} has two positive equilibrium points $P_{1,2}=(u_{1,2},u_{1,2}+C)$ if $\Delta>0$; and
\item System \eqref{eq:05} has one positive equilibrium point $P_3=((u_3,u_3+C))$ (order two) if $\Delta=0$.
\end{enumerate}

\section{Main Results}\label{result}
In this section, we discuss the stability of the equilibrium points and their bifurcations.
\begin{theorem}\label{bounded}
The region $\Phi=\{(u,v),\ 0\leq u\leq1,\ 0\leq v\leq 1+C\}$ is an invariant region and attracts all trajectories starting in the first quadrant.
\end{theorem}
\begin{proof}
We follow the proof of \cite{arancibia3} where a Holling--Tanner model with strong Allee effect is studied. The main difference between the system studied in \cite{arancibia3} and system \eqref{eq:05} is that the equilibrium points are located in $\Phi=\{(u,v),\ 0\leq u\leq1,\ 0\leq v\leq1+C\}$. However, the invariant region $\Gamma$ is the same invariant region showed in \cite{arancibia3} and the system is also a Kolmogorov type. Therefore, trajectories enter into $\Gamma$ and remain in $\Gamma$, see Figure \ref{Fig3}. Moreover, trajectories inside $\Lambda=\{(u,v),\ u>1,\ 0<v<u+C\}$ enter into $\Phi$ or the region $\Theta=\{(u,v),\ u>1,\ v\geq u+C\}$ since $du/d\tau<0$ and $dv/d\tau>0$, see $\Lambda$ and $\Theta$ in Figure \ref{Fig3}. The $u$-component of trajectories in $\Theta$ are non-increasing as time increases and then these trajectories enter into $\Gamma\backslash\Phi$. As a result, all trajectories starting outside $\Gamma$ enter into $\Gamma$ and end up in $\Phi$ since if $u<1+C$, then $dv/d\tau<0$.
\end{proof}

\begin{figure}
\begin{center}
\includegraphics[width=8cm]{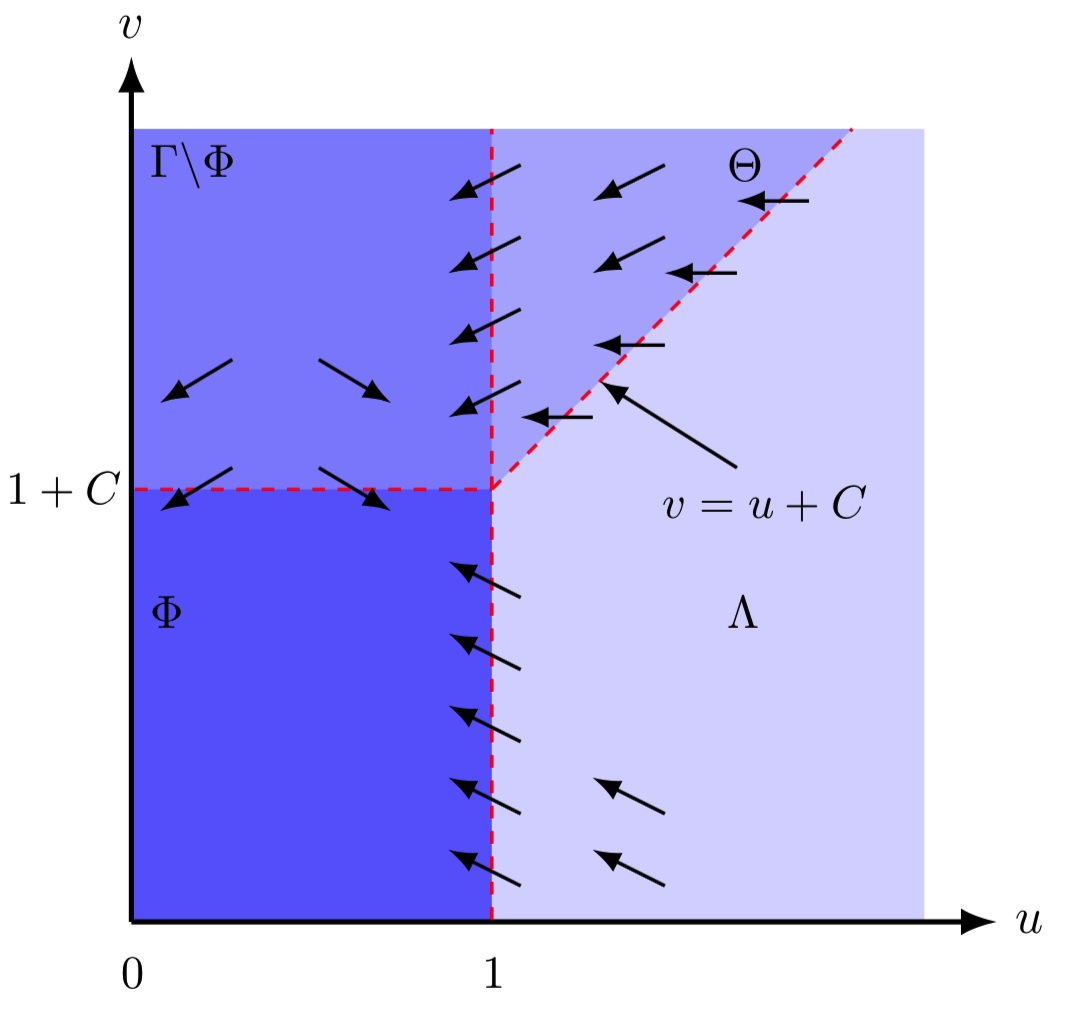}
\end{center}
\caption{Phase plane of system \eqref{eq:05} and its invariant regions $\Phi$ and $\Gamma\backslash\Phi$.}
\label{Fig3}
\end{figure}

\subsection{Nature of equilibrium points}

The Jacobian matrix $J(u,v)$ of system \eqref{eq:05} is
\begin{align}\label{jacobian}
\begin{aligned}
J(u,v)=\begin{pmatrix}
J_{11}(u,v)+J_{12} (u,v) & -uQh(u) \\ 
Sv(B+C+2u-v)  &  S(u+B)(C+u-2v) 
\end{pmatrix},
\end{aligned}
\end{align}
where $J_{11}(u,v)=((1-u)(u-M)-Q(u+B)v)(2u+C)$, $J_{12}(u,v)=(M-2u+1-Qv)(u+C)u$ and $d(u)$ is defined in \eqref{eq:06}.  

\begin{lemma}\label{eqp01}
The equilibrium points $(0,0)$ and $(1,0)$ are saddle points.
\end{lemma}
\begin{proof}
	The Jacobian matrix evaluated at $(0,0)$ gives
	\[ J(0,0)=\begin{pmatrix}
	-CM  & 0 \\ 
	0  &  BCS 
	\end{pmatrix},\]
	with eigenvalues $\lambda_{(0,0)}^1=-CM<0$ and $\lambda_{(0,0)}^2=BCS>0$ and eigenvectors $$\psi_{(0,0)}^1=\begin{pmatrix} 1 & 0 \end{pmatrix}^T~ \text{and}~\psi_{(0,0)}^2=\begin{pmatrix} 0 & 1 \end{pmatrix}^T.$$	
	Similarly, the Jacobian matrix evaluated at $(1,0)$ gives
	\[J(1,0)=\begin{pmatrix}
	(M-1)(C+1)  & -Q(B+1)(C+1) \\ 
	0  &   S(B+1)(C+1)
	\end{pmatrix},\]
	with eigenvalues $\lambda_{(1,0)}^1=S(C+1)(B+1)>0$ and, since $0<M<1$, $\lambda_{(1,0)}^2=(M-1)(C+1)<0$. The associated eigenvectors are $$\psi_{(1,0)}^1=\begin{pmatrix} -Q(B+1)/S(B+1)+1-M & 1 \end{pmatrix}^T~\text{and}~\psi_{(1,0)}^2=\begin{pmatrix} 1 & 0 \end{pmatrix}^T.$$
	Thus, it follows that $(0,0)$ and $(1,0)$ are a saddle points in system \eqref{eq:05}.
\end{proof}

\begin{lemma}\label{eqpM}
The equilibrium point $(M,0)$ is a repeller.
\end{lemma}
\begin{proof}	
	The Jacobian matrix evaluated at $(M,0)$ gives
	\[J(M,0)=\begin{pmatrix}
	-M(M-1)(C+M)  & -MQ(B+M)(C+M) \\ 
	0  &   S(B+M)(C+M)
	\end{pmatrix},\]
	with eigenvalues $\lambda_{(M,0)}^1=M(1-M)(C+M)>0$ and $\lambda_{(M,0)}^2=M(1-M)(C+M)>0$ and eigenvectors $$\psi_{(M,0)}^1=\begin{pmatrix} MQ(B+M)/(M(1-M)-S(B+M)) & 0 \end{pmatrix}^T~\text{and}~\psi_{(M,0)}^2=\begin{pmatrix} 1 & 0 \end{pmatrix}^T.$$
	It follows that $(M,0)$ is a hyperbolic repeller in system \eqref{eq:05}.
\end{proof}

\begin{lemma}\label{eqpC}
If $\Delta\geq0$ \eqref{delta}, then the equilibrium point $(0,C)$ is a local attractor. Moreover, if $\Delta<0$ \eqref{delta}, then $(0,C)$ is a global attractor (for positive initial conditions).
\end{lemma}
\begin{proof}
	The Jacobian matrix evaluated in the point $(0,C)$ is
	\[J(0,C)=\begin{pmatrix}
	-C(M+BQC)  & 0 \\ 
	BCS  &   -BCS
	\end{pmatrix},\]
	with eigenvalues $\lambda_{(0,C)}^1=-C(BCQ+M)<0$ and $\lambda_{(0,C)}^2=-BCS<0$ and eigenvectors $\psi_{(0,C)}^1=\begin{pmatrix} -(M+B(CQ-S))/BS & 1 \end{pmatrix}^T$ and $\psi_{(0,C)}^2=\begin{pmatrix} 0 & 1 \end{pmatrix}^T$. It follows that $(0,C)$ is local attractor in system \eqref{eq:05}.
Moreover, if $\Delta<0$ \eqref{delta}, then $(0,C)$ is the only stable equilibrium point in $\Phi$. Hence, by the Poincar\'e--Bendixson Theorem $(0,C)$ is the unique $\omega$-limit for all trajectories starting in the first quadrant, since by Theorem \ref{bounded} all positive solutions are bounded and eventually end up in $\Gamma$, see Figure \ref{Fig4}.
\end{proof}
\begin{figure}
\centering
\includegraphics[width=9cm]{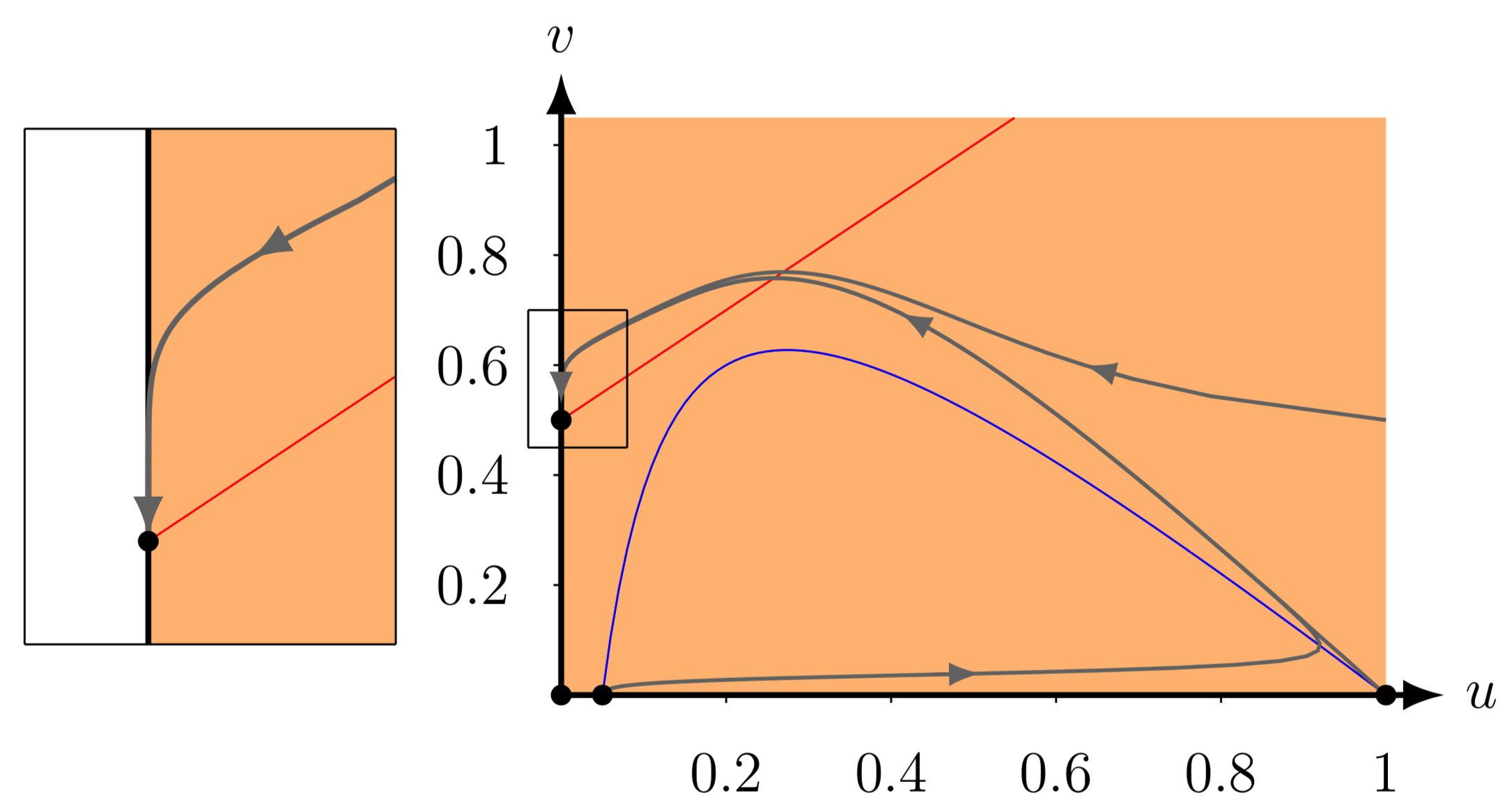}
\caption{For $M=0.05$, $B=0.05$, $C=0.5$, $Q=0.8$, and $S=0.175$, such that $\Delta<0$ \eqref{delta}, the equilibrium point $(0,C)$ is a global attractor for trajectories starting in the first quadrant. The blue (red) curve represents the prey (predator) nullcline.}
\label{Fig4}
\end{figure}

Next, we consider system parameters values such that system \eqref{eq:05} has two equilibrium points in the first quadrant, that is, we assume $\Delta>0$ \eqref{delta}. These equilibrium points lie on the line $v=u+C$ such that $Qh(u)=g(u)$ \eqref{eq:06} and $J_{11}=0$ \eqref{jacobian}. Hence, the Jacobian matrix \eqref{jacobian} at these equilibrium points simplifies to
\begin{equation}\label{J}
J(u_{i},u_{i}+C)=\begin{pmatrix}
J_{12}(u_i,u_i+C)  & -Qu_{i}(u_{i}+B)(u_{i}+C) \\ 
S(u_{i}+B)(u_{i}+C)  &  -S(u_{i}+B)(u_{i}+C) 
\end{pmatrix},\end{equation}
with $J_{12}(u_i,u_i+C)=(M-2u_i+1-Q(u_i+C))(u_i+C)u_i$, $i=1,2$ and $u_i$ given in \eqref{delta}. The determinant and the trace of the Jacobian  matrix \eqref{J} are:\\
\[\begin{aligned}
\det(J(u_{i},u_{i}+C))=&Su_{i}(u_{i}+B)(u_{i}+C)^2(-M+2u_{i}(1+Q)\\
&-1+Q(B+C)),\\
\tr(J(u_{i},u_{i}+C))=&(u_{i}+B)(u_{i}+C)\left(f(u_{i})-C\right),
\end{aligned}\]
where
\begin{align}\label{function}
\begin{aligned}
f(u_{i})=\dfrac{(u_{i}(M-2u_{i}-Qu_{i}+1)-S(B+u_{i}))}{u_{i}Q}.
\end{aligned}
\end{align}
Thus, the sign of the determinant depends on the sign of $-M+2u_{i}(1+Q)-1+Q(B+C)$ and the sign of the trace depends on the sign of $f(u_{i})-C$. Moreover, the eigenvalues of the Jacobian matrix of system \eqref{eq:05} evaluate at $P_1=(u_1,u_1+C)$ are 
\[\lambda_{P_1}^{1,2}=-\dfrac{(u_1+C)\left(BS+u_1(-1-M+S+CQ+u_1(2+Q))\right)\pm \sqrt{p_1(u_1)}}{2}\]
and eigenvectors
\[\psi_{P_1}^{1,2}=\begin{pmatrix} \dfrac{BS+u_1(1+M+S-CQ-u_1(2+Q))\pm \sqrt{p_2(u_1)}}{2S(u_1+B)} \\ 1 \end{pmatrix}\] 
with 
\begin{align*}
p_1(u_1)=&-4S(u_1+B)(u_1BQ+u_1(-1-M+CQ+u_1Q\\
&+u_1(2+Q)))+(BS+u_1(-1-M+S+CQ+(2+Q)u))^2,\\
p_2(u_1)=&(B^2S(-4Qu_1+S)+2BS(1+M+S-CQ-4Qu_1)u_1\\
&+((1+M-CQ)^2-2(-1-M+CQ+B(2+Q)+2Qu_1)S\\
&+S^2)u_1^2- 2(2 + Q)(1+M+S-CQ)u_1^3+(2+Q)^2u_1^4).
\end{align*} 
Note that the first element of $\psi_{P_1}^{1}>0$ since $1+M-CQ>u_1(2+Q)$. Similarly, it turns out that $\psi_{P_1}^{2}>0$. This gives the following results. 
\begin{lemma}\label{p1}
If $\Delta>0$ \eqref{delta}, then the equilibrium point $P_1$ is a saddle point.
\end{lemma}
\begin{proof}
	Evaluating $-M+2u(1+Q)-1+Q(B+C)$ at $u_1$ gives:\\
	\[\begin{aligned}
	-M+2u_1(1+Q)-1+Q(B+C) &=-\sqrt{\Delta}<0.
	\end{aligned}\]
	Hence $\det(J(P_1))<0$ and $P_1$ is thus a saddle point, see Figure \ref{Fig5}.
\end{proof}
\begin{lemma}\label{p2}
If $\Delta>0$ \eqref{delta}, then the equilibrium point $P_2$ is:
	\begin{enumerate}\renewcommand{\labelenumi}{(\roman{enumi})}
	\item a repeller if $0<C<C_{H}=f(u_2)$; and
	\item an attractor if $C>C_{H}$,
\end{enumerate}
with $f$ defined in \eqref{function}.
\end{lemma}
\begin{proof}
	Evaluating $-M+2u(1+Q)-1+Q(B+C)$ at $u_2$ gives:\\
	\[\begin{aligned}
	-M+2u_2(1+Q)-1+Q(B+C) =\sqrt{\Delta}>0.
	\end{aligned}\]
	Hence $\det(J(P_2))>0$. Evaluating $f(u)-C$ at $u=u_2$ gives
	\[f(u_2)-C=\dfrac{(u_2(M-2u_2-Qu_2+1)-S(B+u_2))}{u_2Q}-C_{H}.\]
	Therefore, the sign of the trace, and thus the behaviour of $P_2$, depends on the parity of  $f(u_2)-C_{H}$, see Figure \ref{Fig5}.
\end{proof} 
If $\Delta>0$ and $C>C_{H}$, then system \eqref{eq:05} has two stable equilibrium points $(0,C)$ and $P_2$. Furthermore, if $C=C_{H}$, then $\tr(J(P_2))=0$ and 
$P_2$ undergoes a Hopf bifurcation \cite{chicone}.
\begin{figure}
\centering
\includegraphics[width=12cm]{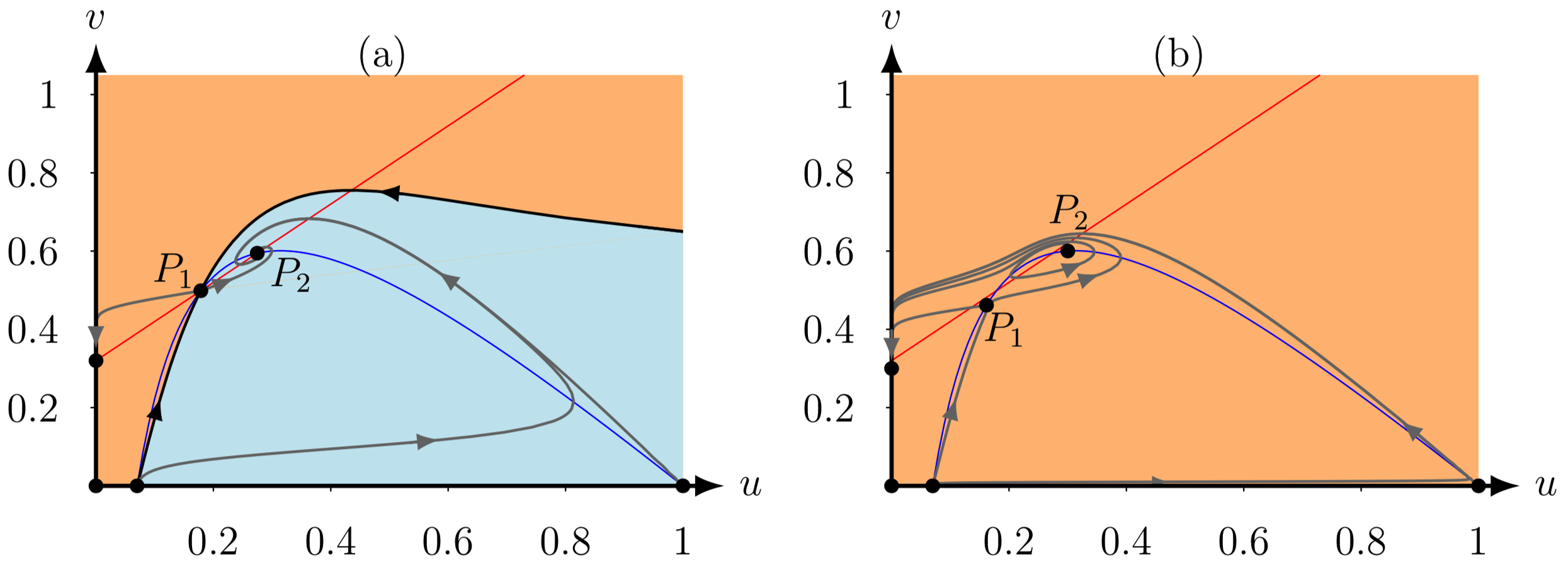}
\caption{Let the system parameter $(M,B,C,Q)=(0.07,0.0645,0.32,0.736)$ be such that $\Delta>0$ \eqref{delta}. (a) If $S=0.15$ such that $C<C_{H}$, then the equilibrium point $P_2$ is stable. (b) If $S=0.05$ such that $C>C_{H}$, then the equilibrium point $P_2$ is unstable. The blue (red) curve represents the prey (predator) nullcline. The orange (light blue) region represents the basin of attraction of the equilibrium point $(0,C)$ ($P_2$). Note that the same color conventions are used in the upcoming figures.}
\label{Fig5}
\end{figure}

Finally, if $\Delta=0$ \eqref{delta} then the equilibrium points $P_1$ and $P_2$ collapse and system \eqref{eq:05} has a unique equilibrium point in the first quadrant.

\begin{lemma}\label{p1=p2}
If $\Delta=0$ \eqref{delta}, then the equilibrium point $P_3$ is:
	\begin{enumerate}\renewcommand{\labelenumi}{(\roman{enumi})}
	\item a saddle-node attractor if $C>C_{SN}=f(u_3)$; and
	\item a saddle-node repeller if $C<C_{SN}$,
\end{enumerate}
with $f$ defined in \eqref{function}.
\end{lemma}
\begin{proof}
	Evaluating $-M+2u(1+Q)-1+Q(B+C)$ at $u=u_3$ gives:\\
	\[\begin{aligned}
	-M+2u_3(1+Q)-1+Q(B+C) =0.
	\end{aligned}\]
	Hence $\det(J(P_3))=0$. Evaluating $f(u)-C$ at $u=u_3$ gives
	\[f(u_3)-C=\dfrac{(u_3(M-2u_3-Qu_3+1)-S(B+u_3))}{u_3Q}-C.\]
	Therefore, the sign of the trace, and thus the behaviour of $P_3$, depends on the parity of  $f(u_3)-C$, see Figure \ref{Fig6}.
\end{proof}
\begin{figure}
\begin{center}
\includegraphics[width=12cm]{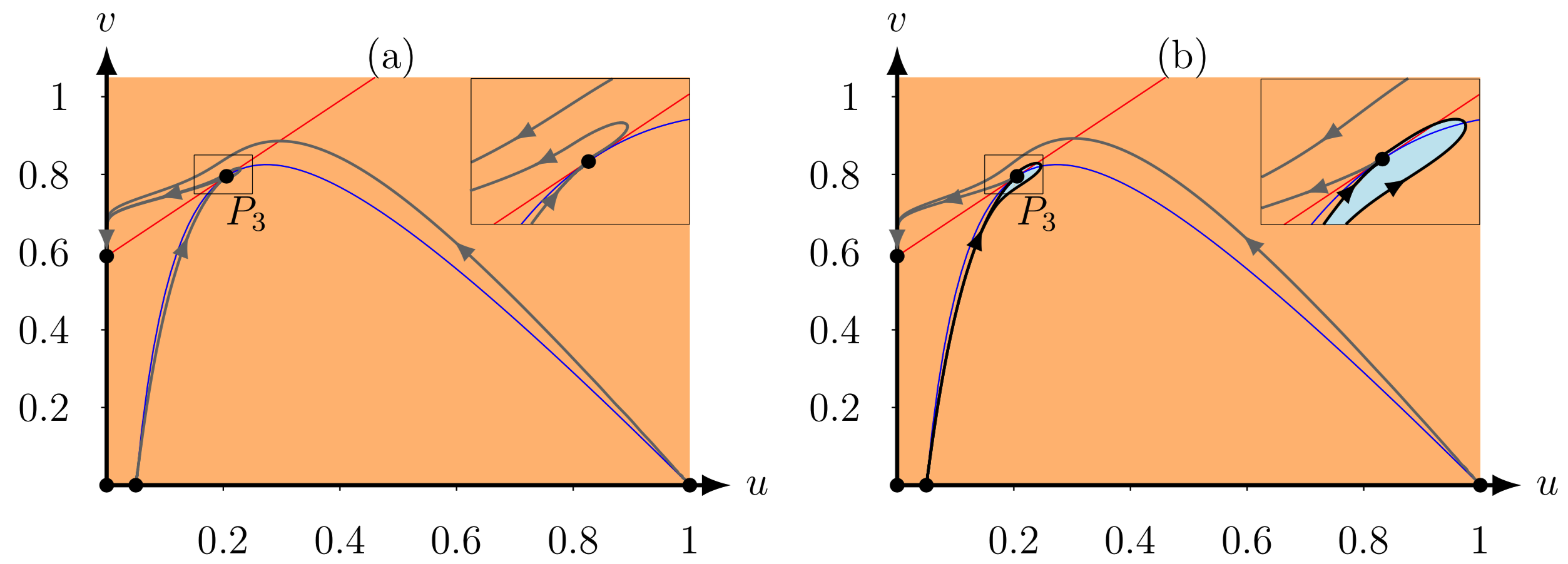}
\caption{If $M=0.05$, $B=0.05$, $S=0.125$ and $Q=0.60821818$, then $\Delta=0$. Therefore, the equilibrium point $P_3$ is (a) a saddle-node repeller if $C>C_{SN}$ and (b) a saddle-node attractor if $C<C_{SN}$.}
\label{Fig6}
\end{center}
\end{figure}

\subsection{Bifurcation analysis}
In this section we present some of the possible bifurcation scenarios when $\Delta=0$ \eqref{delta} in system \eqref{eq:05}.

\begin{theorem}
If $\Delta=0$ \eqref{delta}, then by changing $Q$ system \eqref{eq:05} experiences a saddle-node bifurcation at the equilibrium point $P_3$.
\end{theorem}
\begin{proof}
In order to prove the saddle-node bifurcation at $P_3=(u_3,u_3+C)$ with $u_3=(1+M-Q(B+C))/(2(Q+1))$ we follow the Sotomayor's Theorem \cite{perko}. First, if we consider $\Delta=0$ then system \eqref{eq:05} has one positive equilibrium point $P_3=(u_3,u_3+C)$. Moreover, in Lemma~\ref{p1=p2} we showed that if $\Delta=0$, then $\det(J(P_3))=0$. So, $\lambda=0$ is an eigenvalue of the Jacobian matrix $J(P_3)$ with eigenvector $U=\begin{pmatrix} 1 & 1 \end{pmatrix}^T$. Furthermore, we denote $W$ as the eigenvector corresponding to the eigenvalue $\lambda=0$ of the Jacobian matrix $J(P_3)^T$
$$W=\begin{pmatrix}
-\dfrac{2S(Q+1)}{Q(1+M-Q(B+C))} & 1
\end{pmatrix}^T$$
	
The vector form of system \eqref{eq:05} is given by
\begin{align*}
F((u,v);Q) &=\begin{pmatrix}
u(u+C)((1-u)(u_3-M) -Q(u+B)v)\\ 
Sv(u+B)(u-v+C)
\end{pmatrix},
\end{align*}
then differentiating $F$ with respect to the bifurcation parameter $Q$ at $P_3$ gives
\begin{align*}
F_Q(P_3;Q)&=\begin{pmatrix}
-u_3(u_3+B)(u_3+C)^2\\ 
0
\end{pmatrix},
\end{align*}
with $-u_3(u_3+B)(u_3+C)^2=\dfrac{1}{16(1+Q)^4}(1+M-BQ-CQ)(1+M+BQ-CQ+2B)(-2C-M+BQ-CQ-1)^2.$

Therefore,
\begin{align*}
&W\cdot F_Q(P_3;Q)=\\
 &-\dfrac{S(2B+1+M)(1+M-BQ-CQ)(2C+M-BQ+CQ+1)^2}{8Q(Q+1)^3(M-Q(B+C)+1)}-\\
&\dfrac{SQ(B-C)(1+M-BQ-CQ)(2C+M-BQ+CQ+1)^2}{8Q(Q+1)^3(M-Q(B+C)+1)}<0,
\end{align*}
since we assumed $\Delta=0$ and $u_3>0$. 

Next, we analyse the expression $W \cdot [D^2F(P_3;Q)(U,U)]$.
Therefore, we first compute the Hessian matrix at the equilibrium point $P_3$
\[\begin{aligned}	
D^2F(P_3;Q)(U,U)=
 \begin{pmatrix}
-2(C(2-M)+3-2M+Q(3(2+B)+C(3+B)))\\ 
2CS
\end{pmatrix}&. 
\end{aligned}\]
Hence, since $M \in (0,1)$, we get
\[\begin{aligned}
&W \cdot [D^2F(P_3;Q)(U,U)]=\\
&2CS+\dfrac{4S(Q+1)(C(2-M)+2(1-M)+1+6Q+BCQ+3BQ+3CQ)}{Q(1+M-Q(B+C))}\\
&>0 \,,
\end{aligned}\]
again since we assumed $\Delta=0$ and $u_3>0$. 

Thus, the conditions of Sotomayor's Theorem \cite{perko} are satisfied. Hence, system \eqref{eq:05} experiences a saddle-node bifurcation at the equilibrium point $P_3$.
\end{proof}

If $\Delta=0$ \eqref{delta} and $C=C_{SN}=f(u_3)$, then the equilibrium points collapse and system \eqref{eq:05} has one positive equilibrium point $P_3$. This equilibrium point is a cusp point given that a non-degeneracy condition is met. To show that, we first translate the equilibrium point $P_3=(u_3,u_3+C)$ to the origin by setting $X=u-u_3$ and $Y=v-u_3-C$ and expand system \eqref{eq:05} in a power series around the origin. System \eqref{eq:05} can now be written as
\begin{equation}\label{BTe1}
\begin{aligned}
\dfrac{dX}{d\tau}=& \dfrac{1}{Q^2}S(S+BQ)(S+CQ)X-\dfrac{1}{Q^2}S(S+BQ)(S+CQ)Y+\dfrac{1}{4(Q+1)^2}(BCQ^3\\
&-3B^2Q^2-B^2Q^3+BCQ^2+4BCQ+4BMQ+4BQ+2C^2Q-CMQ^2\\
&+CMQ-2CM-CQ^2+CQ-2C+M^2Q-M^2+2MQ-2M+Q\\
&-1) X^2-\dfrac{1}{4(Q+1)^2}Q(-2C-M+BQ-CQ-1)(-M+BQ+CQ\\&-1)XY+\mathcal{O}(|X,Y|^3),\\
\dfrac{dY}{d\tau}=&\dfrac{1}{Q^2}S(S+BQ)(S+CQ)X+\dfrac{1}{Q^2}S(S+BQ)(S+CQ)Y+\dfrac{1}{2(Q+1)}S(2C\\
&+M-BQ+CQ+1)X^2+S(B-C)XY +\dfrac{S}{2(Q+1)}(2B+M+BQ\\
&-CQ+1)Y^2+\mathcal{O}(|X,Y|^3).
\end{aligned}
\end{equation}
Making the affine transformation $$U=X~\text{and}~V=\dfrac{1}{Q^2}S(S+BQ)(S+CQ)X-\dfrac{1}{Q^2}S(S+BQ)(S+CQ)Y$$ system \eqref{BTe1} becomes 
\begin{equation}\label{BTe2}
\begin{aligned}
\dfrac{dU}{d\tau}=&V-\dfrac{1}{4(Q+1)^2}(2C+2M-2BQ^2+CQ^2-2C^2Q+3B^2Q^2+2B^2Q^3-2C^2Q^2\\
&-C^2Q^3+2CM-4BQ+CQ+M^2-3BCQ^2-BCQ^3-2BMQ^2+CMQ^2\\
&-4BCQ-4BMQ+CMQ+1)U^2+\dfrac{1}{4S(S+BQ)(S+CQ)(Q+1)^2}(Q^3(M\\
&-BQ-CQ+1)(2C+M-BQ+CQ+1)UV+\mathcal{O}(|U,V|^3),\\
\dfrac{dV}{d\tau}=&\dfrac{1}{Q^4}S(S+BQ)(S+CQ)(Q^2+S^3+BQS^2+CQS^2+BCQ^2S)U^2\\
&+\dfrac{1}{4Q^2(Q+1)^2}(4S^2(S+BQ)(S+CQ)(Q+1)(3B-C+M+2BQ-2CQ\\
&+1)+Q^3(M-BQ-CQ+1)(2C+M-BQ+CQ+1))UV-\dfrac{1}{2(Q+1)}(2B\\
&+M+BQ-CQ+1)V^2+\mathcal{O}(|U,V|^3).
\end{aligned}
\end{equation}
By Lemma $3.1$ presented in \cite{xiao2} we obtain an equivalent system of \eqref{BTe2} as follows
\begin{equation}\label{BTe3}
\begin{aligned}
\dfrac{dU_1}{d\tau}=&V_1,\\
\dfrac{dV_1}{d\tau}=&L_{20}U_1^2+L_{11}U_1V_1+\mathcal{O}(|U_1,V_1|^3),
\end{aligned}
\end{equation}
with $L_{20}=S(S+BQ)(S+CQ)(Q^2+S^3+BQS^2+CQS^2+BCQ^2S)/Q^4>0$ since all the parameters are positive and $L_{11}=(4S^2(S+BQ)(S+CQ)(Q+1)(3B-C+M+2BQ-2CQ+1)+Q^2(Q(M-BQ-CQ+1)(2C+M-BQ+CQ+1)-2(2C+2M-2BQ^2+CQ^2-2C^2Q+3B^2Q^2+2B^2Q^3-2C^2Q^2-C^2Q^3+2CM-4BQ+CQ+M^2-3BCQ^2-BCQ^3-2BMQ^2+CMQ^2-4BCQ-4BMQ+CMQ+1)))/(4Q^2(Q+1)^2)$. If $L_{11}\neq0$, then $P_3$ is a cusp point of codimension two by the result presented in \cite{perko}.

This is also a necessary condition for system \eqref{eq:05} to undergo a Bogdanov-Takens bifurcation \cite{perko}. One needs to vary two parameters in order to encounter this bifurcation in a structurally stable way and to describe all possible qualitative behaviours nearby \cite{chicone,perko}. The proof of a Bogdanov-Takens bifurcation can be obtained by following \cite{huang} and \cite{xiao2}. In these articles, the authors showed that their system undergoes to a Bogdanov--Takens bifurcation by unfolding the system around the cusp of codimension two. Moreover, by using a series of normal form transformations one can check the non-degeneracy condition. Nowadays, there are several computational methods to find Bogdanov-Takens points. These methods are implemented in software packages such as MATCONT \cite{matcont}. Figure \ref{Fig8} illustrates the Bogdanov-Takens bifurcation which was detected with MATCONT in the $(Q,C)$-plane with parameter values $(M,B,S)=(0.05,0.1,0.071080895)$ fixed.
\begin{figure}
\begin{center}
\includegraphics[width=7cm]{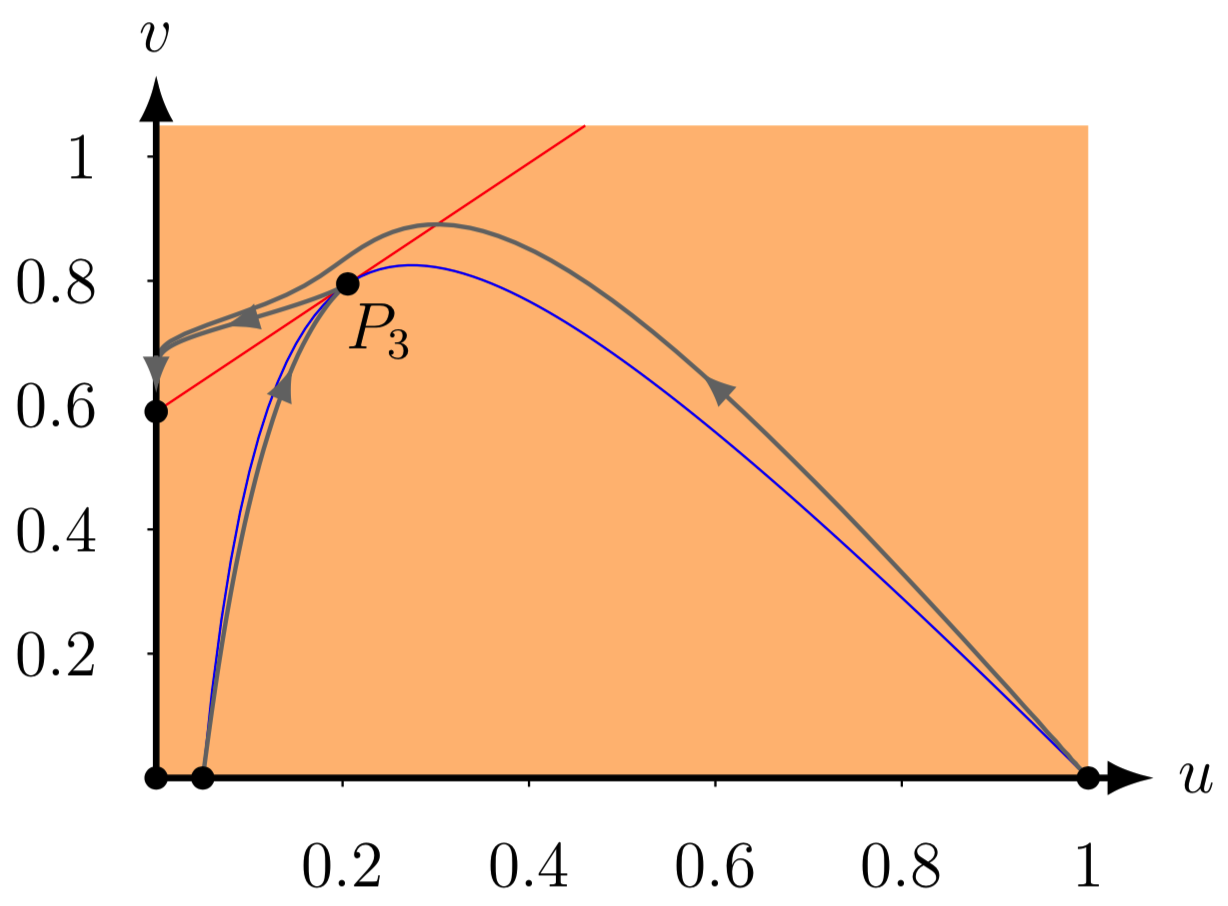}
\caption{For $M=0.05$, $B=0.05$, $C=0.58951256$, $S=0.125$ and $Q=0.60821818$, such that $\Delta=0$ and $f(u_3)=C_{SN}$, the point $(0,C)$ is an attractor and the equilibrium point $P_3$ is a cusp point.}
\label{Fig7}
\end{center}
\end{figure}

\subsection{Basins of attraction}

In this section, we analyse the impact of the modifications of the parameters $C$ and $Q$ on the basins of attraction of the stable equilibrium points of system \eqref{eq:05}. Note that the parameter $C= c/(nK)$ of system \eqref{eq:05} is equivalent to the alternative food source $c$ in system \eqref{eq:04} since the function $\varphi$ is a diffeomorphism preserving the orientation of time. Similarly, the parameter $Q =qnK/r$ of system \eqref{eq:05} is equivalent to the predation rate $q$ in system \eqref{eq:04}. In particular, we consider the system parameters $(B,M,S)=(0.1,0.1,0.157)$\footnote{Note that changing $B$ instead of $Q$ has the same qualitative effect on the basin of attraction, see right pane of Figure \ref{Fig8}.} and vary $Q$ and $C$. For $Q$ and $C$ not too big system \eqref{eq:05} has two positive equilibrium points, namely $P_1$ and $P_2$. The equilibrium points on the axis and $P_1$ do not change stability proven in Lemmas \ref{eqp01}, \ref{eqpM}, \ref{eqpC} and \ref{p1}, while, $P_2$ can be stable or unstable.

\begin{figure}
\begin{center}
\includegraphics[width=12cm]{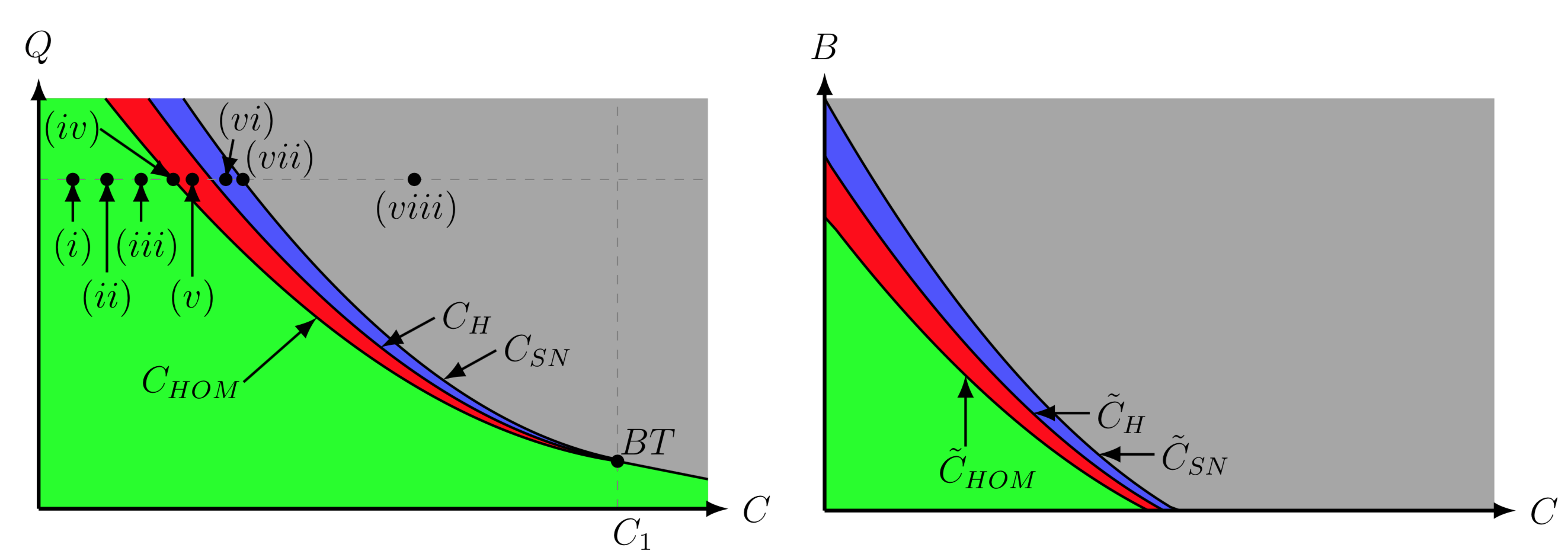}
\end{center}
\caption{The bifurcation diagram of system \eqref{eq:05} for $M=0.05$ and $S=0.071080895$ fixed and created with the numerical bifurcation package MATCONT \cite{matcont}. In the left panel $B=0.1$ fixed and varying $Q$ and $C$ and in the right panel $Q=0.608$ fixed and varying $B$ and $C$. The curve $C_H$ represents the Hopf curve, $C_{HOM}$ represents the homoclinic curve, $C_{SN}$ represents the saddle-node curve, and $BT$ represents the Bogdanov-Takens bifurcation.The corresponding phase planes for the different regions are shown in Figure \ref{Fig9}.}
\label{Fig8}
\end{figure}
In order to study the basins of attraction of the equilibrium points $(0,C)$ and $P_2$ we use the same notation for the (un)stable manifold of the equilibrium point $P_1$ as used in \cite{arancibia7,arancibia3}. That is, we defind $W^{u,s}_{\nearrow}(P_1)$ as the branch of the (un)stable manifold of $P_1$ that goes up to the right and $W^{u,s}_{\swarrow}(P_1)$ as the branch of the (un)stable manifold of $P_1$ that goes down to the left. The branch $W^s_{\nearrow}(P_1)$ is connected with $(M,0)$ and $W^u_{\swarrow}(P_1)$ is connected with $(0,C)$ since the nullclines form a bounding box from which trajectories cannot leave. Furthermore, everything in between of these two branches and the $x$-axis also asymptotes to the equilibrium point $(0,C)$. Therefore, the stable manifold of the saddle point $P_1$ acts as a separatrix curve between the basins of attraction of $P_2$ (when it is stable) and $(0,C)$, see Figure \ref{Fig9}. 

\begin{figure}
\begin{center}
\includegraphics[width=11cm]{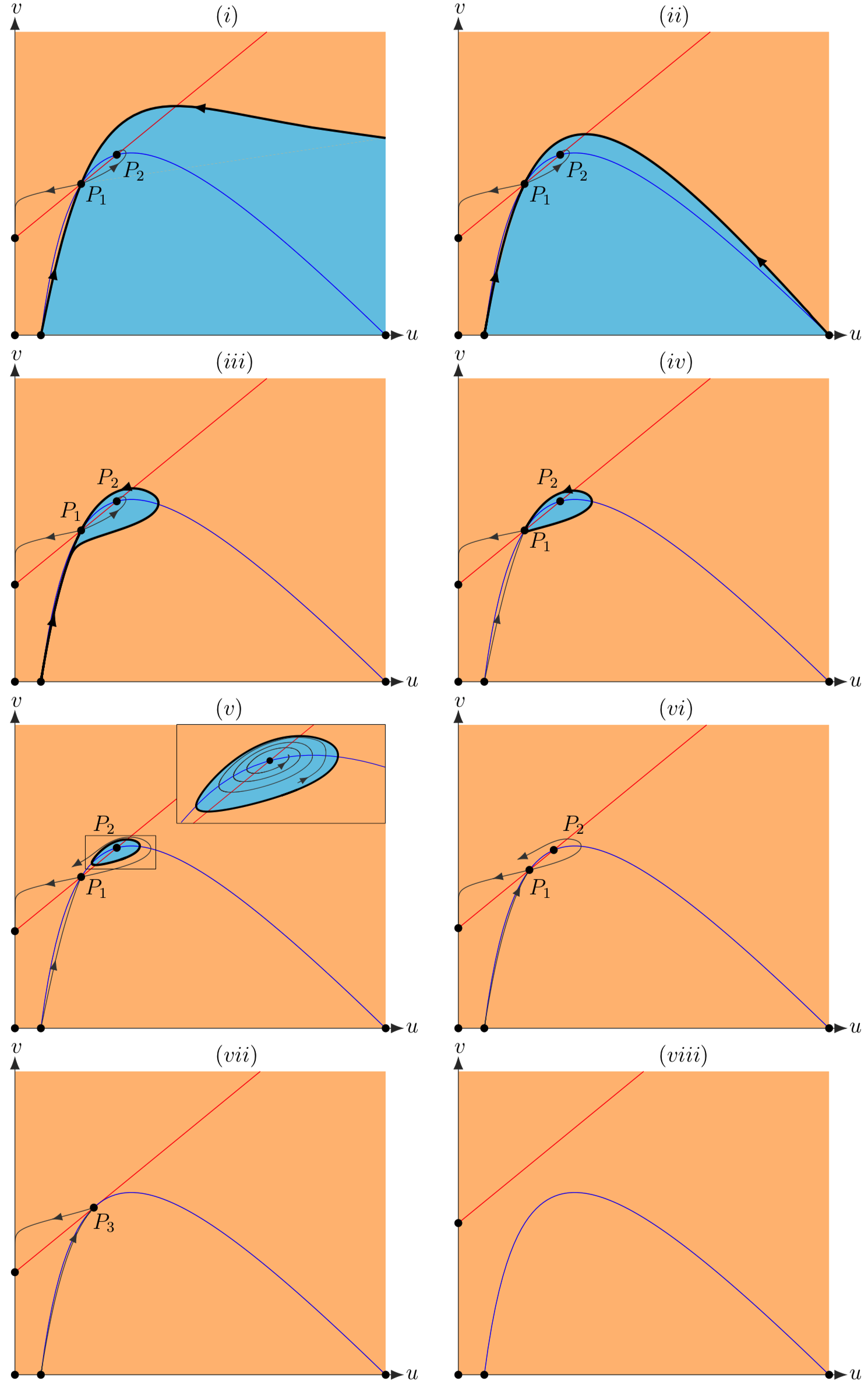}
\end{center}
\caption{The phase planes of system \eqref{eq:05} for $B=0.1$, $M=0.05$, $Q=0.75$ and $S=0.071080895$ fixed and varying $C$. This last parameter impacts the number of equilibrium points of system \eqref{eq:05}. The light blue area in the phase plane represent the basins of attraction of the equilibrium points $P_2$, while the orange area in the phase plane represent the basins of attraction of the equilibrium points $(0,C)$.}
\label{Fig9}
\end{figure}

Considering the invariant region $\Phi$, there are qualitatively six different cases for the boundaries of the basins of the equilibrium points $P_2$ and $(0,C)$, then we get:
\begin{enumerate}[(i)]
\item For $C<C_{HOM}$ such that the equilibrium point $P_2$ in system \eqref{eq:05} is stable, see Lemma \ref{p2} (since $C_{HOM}<C_{H}$). For $C$ small enough $W^s_{\swarrow}(P_1)$  intersects the boundary of $\Phi$. Hence, it forms a separatrix curve in $\Phi$, see panel ($i$) in Figure \ref{Fig9}.  In addition, by increasing $C$ the stable manifold of $P_1$ connects first with $(1,0)$ and then with $(M,0)$, again forming the separatrix curve, see panels ($ii$) and ($iii$) in Figure \ref{Fig9}.
\item For $C=C_{HOM}$, then $W^s_{\swarrow}(P_1)$ connects with $W^u_{\nearrow}(P_1)$, therefore it form a homoclinic curve. Which is the separatrix curve between the basins of attraction of $(0,C)$ and $P_2$, see panel ($iv$) in Figure \ref{Fig9}.
\item For $C_{HOM}<C<C_H$, there is an unstable limit cycle surrounding $P_2$ which acts as a separatrix curve between the basins of attraction of $P_2$ and $(0,C)$. This limit cycle is created around $P_2$ via the Hopf bifurcation at $C=C_{H}$ \cite{gaiko} and terminates via a homoclinic bifurcation at $C=C_{HOM}$, see panel ($v$) in Figure \ref{Fig9}.
\item For $C_H<C<C_{SN}$, the equilibrium point $P_2$ is unstable, see Lemma \ref{p2}, and $(0,C)$ is globally asymptotically stable. Hence, $\Phi$ is the basin of attraction of $(0,C)$, see panel ($vi$) in Figure \ref{Fig9}.
\item For $C=C_{SN}$, the equilibrium points $P_1$ and $P_2$ collapse, see Lemma \ref{p1=p2}. Hence, $\Phi$ is the basin of attraction of $(0,C)$, see panel ($vii$) in Figure \ref{Fig9}.
\item For $C_{SN}<C$, system \eqref{eq:05} dose not have positive equilibrium points, see Lemma \ref{eqpC}. Hence, $\Phi$ is also the basin of attraction of $(0,C)$, see panel ($viii$) in Figure \ref{Fig9}.
\end{enumerate}

\section{Conclusions}\label{con}

In this manuscript, a modified May--Holling--Tanner predator-prey model with multiple Allee effects for the prey and alternative food sources for the predators was studied. Using a diffeomorphism, we transformed the modified May--Holling--Tanner predator-prey model to a topologically equivalent system, system \eqref{eq:05}. Subsequently, we analysed system \eqref{eq:05} and we proved that the equilibrium points $(0,0)$ and $(1,0)$ are saddle points, $(M,0)$ is a repeller and $(0,C)$ is an attractor for all parameter values, see Lemmas \ref{eqp01}, \ref{eqpM} and \ref{eqpC}. Additionally, there exist at most two positive equilibrium points, one of them, $P_1$, is a saddle point, while the other, $P_2$, can be an attractor or a repeller, depending on the trace of its Jacobian matrix. Both equilibrium points can collapse having conditions for a saddle node bifurcations and cusp point \cite{xiao2} (Bogdanov-Takens bifurcation). We also showed the existence of a homoclinic curve, determined by the stable and unstable manifolds of the equilibrium point $P_1$ enclosing the second equilibrium point $P_2$. When the homoclinic breaks it creates a non-infinitesimal limit cycle, see Lemmas \ref{p1}, \ref{p2} and Figure \ref{Fig8}. 

Moreover, by choosing the bifurcation parameters $(C,Q)$, or $(B,C)$, we have obtained significant bifurcation diagrams, see Figure \ref{Fig8}. It follows that -- for a large nondimensionalised predation $Q$ and a small nondimensionalised proportion of alternative food  $C$ -- co-existence is expected. Similarly, when the proportion of nondimensionalised alternative food $C$ is bigger than the proportion of nondimensionalised predation $Q$ co-existence is expected. The bifurcation diagrams and associated phase planes, see Figure \ref{Fig9}, also shows that there exists complexity for system \eqref{eq:05} including the collision of the equilibrium points leading to different type of bifurcation. 

Since the function $\varphi$ is a diffeomorphism preserving \cite{martinez2}the orientation of time, the dynamics of system \eqref{eq:05} are topologically equivalent to the dynamics of system \eqref{eq:04}. Hence, the parameters $(C,Q)$ impact the number of equilibrium points of system \eqref{eq:05} in the first quadrant and change the behaviour of the system, and, as $C = c/(nK)$ and $Q = qnK/r$, the system parameters $(c,n,k,q,r)$ will thus impact the behaviour of system \eqref{eq:04}. Therefore, self-regulation depends on the values of these  parameters. For instance, keeping all parameters fixed, but increasing the alternative food source $c$, one expects to see a change in behavior and dynamics similar to the one shown in Figure \ref{Fig8} and \ref{Fig9}. All these results show that dynamical behavior of system \eqref{eq:04} becomes more complex under the modification of the system parameters when compared to the May--Holling--Tanner model with the strong and weak Allee effect \eqref{allee1} studied in \cite{martinez2}.
\begin{figure}
\begin{center}
\includegraphics[width=9cm]{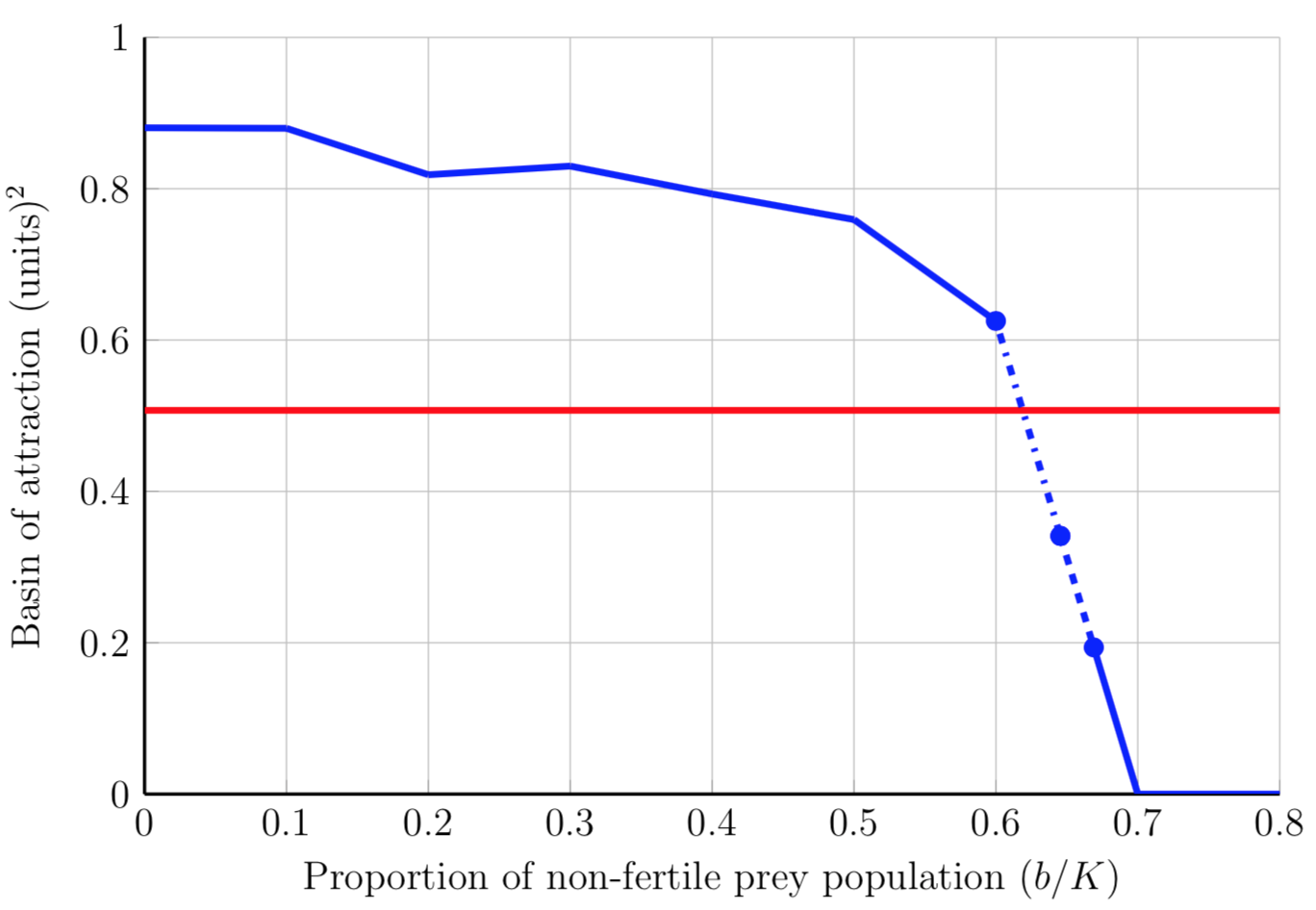}
\end{center}
\caption{The size of the basin of attraction of $p_2$, in units$^2$, of the stable equilibrium point $p_2$ of system \eqref{eq:04} considering strong Allee effect (red line) and multiple Allee effect (blue line) for varying the non-fertile population $b$ and with other system parameters $r=14$, $K=150$, $m=15$, $q=1.08$, $s=1.25$, $n=0.05$ and $c=0.75$ fixed. The blue dotted-dashed line represents the region where the stable manifold of the saddle equilibrium point $p_1$ connects with (K,0) and the blue dashed line represent the region where the equilibrium point $p_2$ is surrounded by an unstable limit cycle.}
\label{Fig10}
\end{figure}

In Figure \ref{Fig1} we showed that the inclusion of a multiple Allee effect changes the shape of the per-capita growth of the prey, and, in particular, reduces the region of depensation. Moreover, we can see in Figure \ref{Fig10} that there exist a critical non-fertile prey population $b_{cr}$ for which the basin of attraction of the equilibrium point $p_2$ of system \eqref{eq:04} is smaller than the basin of attraction of the related $p_2$ of system \eqref{ht1} considering an alternative food source \eqref{food} and with a strong Allee effect \eqref{allee1}. Note that the non-fertile prey population of $60\%$ is realistic. For instance, Monclus {\em et al.\ }\cite{monclus} studied the impact of the different population densities on st{the} \textit{marmot} reproduction. This study used the proportion of fertile female adults which fluctuated between $2.13$ and $19.15\%$ of the total population density. Moreover, we can also conclude that the basin of attraction of the stable positive equilibrium point $p_2$ increases when we reduce the depensation in the model.

Finally, the techniques used in this manuscript show that there is a strong connection between the analysis of the manifold and the basins of attraction of the equilibrium points. This analysis can be applied in population dynamics in order to predict the behaviour in models where there is variation in the non-fertile population. Moreover, we showed that the combination of different techniques such as numerical simulations and bifurcation analysis can be very useful for showing the temporal dynamics in predation interaction.

\section*{References}
\bibliography{References.bib}
\end{document}